\documentclass{amsart}

\usepackage{mystyle}
\title{Bounding \(\chi\) by a fraction of \(\Delta\) for graphs without large cliques}

\author{Marthe Bonamy}
\address{CNRS, LaBRI, Universit\'e de Bordeaux}
\email{marthe.bonamy@u-bordeaux.fr}

\author{Tom Kelly \and Peter Nelson \and Luke Postle}
\address{University of Waterloo}
\email{t9kelly@uwaterloo.ca}
\email{apnelson@uwaterloo.ca}
\email{lpostle@uwaterloo.ca}

\begin{document}
\maketitle

\begin{abstract}
  The greedy coloring algorithm shows that a graph of maximum degree at most $\Delta$ has chromatic number at most $\Delta + 1$, and this is tight for cliques.
  Much attention has been devoted to improving this ``greedy bound'' for graphs without large cliques. 
  Brooks famously proved that this bound can be improved by one if $\Delta \geq 3$ and the graph contains no clique of size $\Delta + 1$.
  Reed's Conjecture states that the ``greedy bound'' can be improved by $k$ if the graph contains no clique of size $\Delta + 1 - 2k$.
  Johansson proved that the ``greedy bound'' can be improved by a factor of $\Omega(\ln(\Delta)^{-1})$ or $\Omega\left(\frac{\ln(\ln(\Delta))}{\ln(\Delta)}\right)$ for graphs with no triangles or no cliques of any fixed size, respectively.
  
  Notably missing is a \textit{linear} improvement on the ``greedy bound'' for graphs without large cliques.
  In this paper, we prove that for sufficiently large $\Delta$, if $G$ is a graph with maximum degree at most $\Delta$ and no clique of size $\omega$, then
  \[\chi(G) \leq 72\Delta\sqrt{\frac{\ln(\omega)}{\ln(\Delta)}}.\]
  This implies that for sufficiently large $\Delta$, if $\omega^{(72c)^2} \leq \Delta$ then $\chi(G) \leq \Delta/c$.

  This bound actually holds for the list-chromatic and even the correspondence chromatic number (also known as DP-chromatic number).  In fact, we prove what we call a ``local version'' of it, a result implying the existence of a coloring when the number of available colors for each vertex depends on local parameters, like the degree and the clique number of its neighborhood.  We prove that for sufficiently large $\Delta$, if $G$ is a graph of maximum degree at most $\Delta$ and minimum degree at least $\ln^2(\Delta)$ with list-assignment $L$, then $G$ is $L$-colorable if for each $v\in V(G)$, \[|L(v)| \geq 72\deg(v)\cdot\min\left\{\frac{\log_2(\chi(v) + 1)}{\ln(\deg(v))}, \frac{\omega(v)\ln(\ln(\deg(v)))}{\ln(\deg(v))}, \sqrt{\frac{\ln(\omega(v))}{\ln(\deg(v))}}\right\},\] where $\chi(v)$ denotes the chromatic number of the neighborhood of $v$ and $\omega(v)$ denotes the size of a largest clique containing $v$.  This simultaneously implies the linear improvement over the ``greedy bound'' and the two aforementioned results of Johansson.
\end{abstract}
%%% Local Variables:
%%% mode: latex
%%% TeX-master: "../main.tex"
%%% End:

\section{Introduction}

Let $G$ be a graph, and for each $v\in V(G)$, let $L(v)$ be a set which we call the \textit{available colors} for $v$.  If each set $L(v)$ is non-empty, then we say that $L$ is a \textit{list-assignment} for $G$.  If $k$ is a positive integer and $|L(v)|\geq k$ for every $v\in V(G)$, then we say that $L$ is a \textit{$k$-list-assignment} for $G$.  An \textit{$L$-coloring} of $G$ is a mapping $\phi$ with domain $V(G)$ such that $\phi(v)\in L(v)$ for every $v\in V(G)$ and $\phi(u)\neq\phi(v)$ for every pair of adjacent vertices $u,v\in V(G)$.  We say that a graph $G$ is \textit{$k$-list-colorable}, or \textit{$k$-choosable}, if $G$ has an $L$-coloring for every $k$-list-assignment $L$.  If $L(v) = \{1,\dots, k\}$ for every $v\in V(G)$, then we call an $L$-coloring of $G$ a \textit{$k$-coloring}, and we say $G$ is \textit{$k$-colorable} if $G$ has a $k$-coloring.  The \textit{chromatic number} of $G$, denoted $\chi(G)$, is the smallest $k$ such that $G$ is $k$-colorable.  The \textit{list-chromatic number} of $G$, denoted $\listchi(G)$, is the smallest $k$ such that $G$ is $k$-list-colorable.

In 1996, Johansson~\cite{J96-tri} famously proved that if $G$ is a triangle-free graph of maximum degree at most $\Delta$, then $\listchi(G) = O\left(\frac{\Delta}{\ln(\Delta)}\right)$.  Determining the best possible value of the leading constant in this bound is of general interest.  The best known lower bound, using $\Delta$-regular graphs, is $\frac{\Delta}{2\ln(\Delta)}$.  In 1995, Kim~\cite{K95} proved that the upper bound holds with a leading constant of $1 + o(1)$ for graphs of girth at least five.  In 2015, Pettie and Su~\cite{PS15} improved the leading constant in the upper bound for triangle-free graphs to $4 + o(1)$, and in 2017, Molloy~\cite{M17}, in the following theorem, improved it to $1 + o(1)$, matching the bound of Kim.
\begin{theorem}[\cite{M17}]\label{tri-free bound}
  If $G$ is a triangle-free graph of maximum degree at most $\Delta$, then
  \begin{equation*}
    \listchi(G) \leq (1 + o(1))\frac{\Delta}{\ln(\Delta)}.
  \end{equation*}
\end{theorem}

Johansson~\cite{J96-Kr} also proved that for any fixed $\omega\geq 4$, if $G$ is a graph of maximum degree at most $\Delta$ with no clique of size greater than $\omega$, then $\listchi(G) = O\left(\frac{\Delta\ln(\ln(\Delta))}{\ln(\Delta)}\right)$; however, the proof was never published.  Molloy~\cite{M17} proved the following stronger result, which holds even when $\omega$ is not fixed.

\begin{theorem}[\cite{M17}]\label{Kr-free bound}
  If $G$ is a graph of maximum degree at most $\Delta$ with no clique of size greater than $\omega$, then
  \begin{equation*}
    \listchi(G)\leq 200\omega\frac{\Delta\ln(\ln(\Delta))}{\ln(\Delta)}.
  \end{equation*}
\end{theorem}

In 1998, Reed~\cite{R98} conjectured the following, sometimes referred to as ``Reed's $\omega, \Delta, \chi$ Conjecture.''

\begin{conjecture}[\cite{R98}]\label{reeds conjecture}
  If $G$ is a graph of maximum degree at most $\Delta$ with no clique of size greater than $\omega$, then
  \begin{equation*}
    \chi(G) \leq \left\lceil\tfrac{1}{2}(\Delta + 1 + \omega)\right\rceil.
  \end{equation*}
\end{conjecture}
It is possible that Conjecture~\ref{reeds conjecture} is also true for the list-chromatic number.  As evidence for his conjecture, Reed~\cite{R98} proved the following.

\begin{theorem}[\cite{R98}]\label{epsilon reeds}
  There exists $\varepsilon > 0$ such that the following holds.  If $G$ is a graph of maximum degree at most $\Delta$ with no clique of size greater than $\omega$, then
  \begin{equation*}
    \chi(G) \leq \left\lceil(1 - \varepsilon)(\Delta + 1) + \varepsilon\omega\right\rceil.
  \end{equation*}
\end{theorem}

Note that Theorem~\ref{epsilon reeds} holds for $\varepsilon = \frac{1}{2}$ if and only if Conjecture~\ref{reeds conjecture} is true.
In 2016, Bonamy, Perrett, and Postle~\cite{BPP17} proved that Theorem~\ref{epsilon reeds} holds for $\varepsilon = \frac{1}{26}$ when $\Delta$ is sufficiently large.  In 2017, Delcourt and Postle~\cite{DP17} proved that Theorem~\ref{epsilon reeds} holds for the list-chromatic number for $\varepsilon = \frac{1}{13}$ when $\Delta$ is sufficiently large.
Results from Ramsey Theory imply that Theorem~\ref{epsilon reeds} is not true for any value of $\varepsilon > \frac{1}{2}$; for example, Spencer~\cite{S77} showed the existence of a graph on $n$ vertices with independence number 2 (and thus chromatic number at least $n/2$) such that every clique has size at most $n^{\frac{1}{2} + o(1)}$.  The blowup of a 5-cycle, i.e.\ the Cartesian product of a clique and a 5-cycle, also demonstrates that Theorem~\ref{epsilon reeds} is not true when $\varepsilon > \frac{1}{2}$, and even that the rounding in Conjecture~\ref{reeds conjecture} is necessary.

Theorem~\ref{Kr-free bound} implies Conjecture~\ref{reeds conjecture} when $\omega = o\left(\frac{\ln(\Delta)}{\ln(\ln(\Delta))}\right)$.  It is natural to ask if a bound stronger than that of Conjecture~\ref{reeds conjecture} can be proved if $\omega = o(\Delta)$ even if $\omega \geq \frac{\ln(\Delta)}{200\ln(\ln(\Delta))}$.  Spencer's result implies that the bound can not be improved if $\omega = \Omega(\Delta^{1/2})$.  Considering this, we were motivated to answer the following question.

\begin{question}\label{chi bound question}
  Does there exist a function $f : \mathbb R\rightarrow \mathbb R$ such that, for every $c > 1$ and every graph $G$ of maximum degree at most $\Delta$ with no clique of size greater than $\Delta^{1/f(c)}$, we have $\listchi(G) \leq \Delta/c$?
\end{question}

Our first result in this paper is the following theorem.
\begin{theorem}\label{low omega thm}
  If $G$ is a graph of maximum degree at most $\Delta$ with no clique of size greater than $\omega$, then
  \begin{equation*}
    \listchi(G) = O\left(\Delta\sqrt{\frac{\ln(\omega)}{\ln(\Delta)}}\right).
  \end{equation*}
\end{theorem}

Theorem~\ref{low omega thm} answers Question~\ref{chi bound question} in the affirmative with a function $f(c)$ that is quadratic in $c$;  for large enough $\Delta$, the function $f(c) = (72c)^2$ suffices.
Determining the best possible function $f$ that confirms Question~\ref{chi bound question} would be very interesting.
As mentioned, Spencer~\cite{S77} showed that $f(2)\geq 2$.  This result actually provides a lower bound on $f(c)$ that is linear in $c$.  Spencer~\cite{S77} proved that the Ramsey number $R(c, \omega)$ is at least $\Omega\left(\left(\omega/\ln(\omega)\right)^{\frac{c+1}{2}}\right)$ as $\omega\rightarrow\infty$ for fixed $c\geq 3$.  Therefore there exists a graph $G$ on $n$ vertices with no independent set of size $c$ (and thus chromatic number at least $n/(c - 1)$) and no clique of size $\omega$ where $n$ is at least $\omega^{\frac{c + 1}{2} - o(1)}$.  Since the maximum degree of a graph is at most its number of vertices, it follows that $f(c) \geq c/2 + 1$ if $c\in\mathbb N$.

The bound of Spencer~\cite{S77} was improved by Kim in~\cite{K95ramsey} for $c = 3$ by a factor of $\ln \omega$ (matching the upper bound of Ajtai, Koml\'os, and Szemer\'edi~\cite{AKS80} up to a constant factor), by Bohman in~\cite{B09} for $c = 4$ by a factor of $\sqrt{\ln\omega}$, and by Bohman and Keevash in~\cite{BK10} for $c\geq 5$ by a factor of $\ln^{\frac{1}{c - 2}}\omega$, but these improvements do not change the resulting lower bound on $f(c)$.

\subsection{Local Versions}

We actually prove a result much stronger than Theorem~\ref{low omega thm}.  One might wonder if the bounds on $|L(v)|$ supplied by Theorems~\ref{tri-free bound}, \ref{Kr-free bound}, \ref{epsilon reeds}, and \ref{low omega thm} can be relaxed to depend on local parameters, such as the degree of the vertex $v$, or the size of a largest clique containing $v$, rather than the global parameters $\Delta$ and $\omega$.  To that end, for a vertex $v$, we let $\deg(v)$ denote the degree of $v$, $\omega(v)$ denote the size of a largest clique containing $v$, and $\chi(v)$ denote the chromatic number of the neighborhood of $v$.

We are interested in proving that a graph $G$ is $L$-colorable whenever every vertex $v$ satisfies $|L(v)| \geq f(v)$ where $f(v)$ depends on parameters such as $\deg(v)$ and $\omega(v)$.  The archetypal example is the classical theorem of Erd\H os, Rubin, and Taylor \cite{ERT79} that a graph is \textit{degree-choosable} (meaning $L$-colorable for any list-assignment $L$ satisfying $|L(v)| = \deg(v)$ for every vertex $v$) unless every block is a clique or an odd cycle.  We call such a Theorem a ``local version.''  Our main result implies local versions of Theorems~\ref{tri-free bound}, \ref{Kr-free bound}, and \ref{low omega thm} simultaneously, although we do not match the leading constant in Theorem~\ref{tri-free bound}.

In fact, we prove the theorem for \textit{correspondence coloring}, a generalization of list-coloring introduced by Dvo\v{r}\'ak and Postle \cite{DP15} in 2015, and also known as \emph{DP-coloring}. We provide a definition in Section~\ref{prelim section}; the theorem as stated below can also be read as if $L$ is a list assignment. 

\begin{theorem}\label{mixed local thm}
  For all sufficiently large $\Delta$ the following holds.  Let $G$ be a graph of maximum degree at most $\Delta$ with correspondence assignment $(L, M)$.
  For each $v\in V(G)$, let
  \begin{equation*}
    f(v) = 72\cdot\min\left\{\sqrt{\frac{\ln(\omega(v))}{\ln(\deg(v))}}, \frac{\omega(v)\ln(\ln(\deg(v)))}{\ln(\deg(v))}, \frac{\log_2(\chi(v) + 1)}{\ln(\deg(v))}\right\}.
  \end{equation*}
  If for each $v\in V(G)$,
  \begin{equation*}
    |L(v)| \geq \deg(v)\cdot f(v)
  \end{equation*}
  and $\deg(v) \geq \ln^2(\Delta)$, then $G$ is $(L, M)$-colorable.
\end{theorem}

Recently, Bernshteyn \cite{B17} proved that Theorems~\ref{tri-free bound} and \ref{Kr-free bound} hold for the \textit{correspondence chromatic number}, which is always at least as large as the list-chromatic number.  Our Theorem~\ref{mixed local thm} implies that ``local versions'' of these theorems are true for correspondence coloring, as follows.

\begin{corollary}\label{local tri-free}
  For some constant $C$ the following holds.  If $G$ is a triangle-free graph of maximum degree at most $\Delta$ with correspondence assignment $(L, M)$ such that for each $v\in V(G)$,
  \begin{equation*}
    |L(v)| \geq C\frac{\deg(v)}{\ln(\deg(v))},
  \end{equation*}
  and $\deg(v) \geq \ln^2(\Delta)$, then $G$ is $(L, M)$-colorable.
\end{corollary}

\begin{corollary}\label{local Kr-free}
  For some constant $C$ the following holds.  If $G$ is a graph of maximum degree at most $\Delta$ with correspondence assignment $(L, M)$ such that for each $v\in V(G)$,
  \begin{equation*}
    |L(v)| \geq C\deg(v)\frac{\omega(v)\ln(\ln(\deg(v)))}{\ln(\deg(v))},
  \end{equation*}
  and $\deg(v) \geq \ln^2(\Delta)$, then $G$ is $(L, M)$-colorable.
\end{corollary}

We also derive the following ``local version'' of a result of Johansson \cite{J96-Kr} on graphs that are \textit{locally $r$-colorable}, meaning the neighborhood of every vertex is $r$-colorable.
\begin{corollary}
  For some constant $C$ the following holds.  If $G$ is a locally $r$-colorable graph of maximum degree at most $\Delta$ with correspondence assignment $(L, M)$ such that for each $v\in V(G)$,
  \begin{equation*}
    |L(v)| \geq C\deg(v)\frac{\log_2(r + 1)}{\ln(\deg(v))},
  \end{equation*}
  and $\deg(v) \geq \ln^2(\Delta)$, then $G$ is $(L, M)$-colorable.
\end{corollary}

Of course, Theorem~\ref{mixed local thm} also implies a ``local version'' of Theorem~\ref{low omega thm}, as follows.
\begin{corollary}\label{local low omega}
  For some constant $C$ the following holds.  If $G$ is a graph of maximum degree at most $\Delta$ with correspondence assignment $(L, M)$ such that for each $v\in V(G)$,
  \begin{equation*}
    |L(v)| \geq C\deg(v)\sqrt{\frac{\ln(\omega(v))}{\ln(\deg(v))}},
  \end{equation*}
  and $\deg(v) \geq \ln^2(\Delta)$, then $G$ is $(L, M)$-colorable.
\end{corollary}

We now argue that Theorem~\ref{low omega thm} follows from Corollary~\ref{local low omega}.
\begin{proof}[Proof of Theorem~\ref{low omega thm} assuming Corollary~\ref{local low omega}]
  Let $G$ be a graph of maximum degree at most $\Delta$ with no clique of size greater than $\omega$.  We may assume that $G$ has minimum degree at least one.  If $G$ has minimum degree at least $\ln^2(\Delta)$, then Corollary~\ref{local low omega} implies $\listchi(G) \leq C\Delta\sqrt{\frac{\ln(\omega)}{\ln(\Delta)}}$, as desired.  Otherwise, we use the following standard procedure to obtain a graph of larger minimum degree containing $G$ as a subgraph.  We duplicate the graph $G$, and we add an edge between each vertex of minimum degree and its duplicate.  Note that the minimum degree is increased by one, and that for every vertex $v$, the size of a largest clique containing $v$ in the new graph does not increase.  We repeat this procedure until we obtain a graph $G'$, having $G$ as a subgraph, and with minimum degree at least $\ln^2(\Delta)$. The result now follows by applying Corollary~\ref{local low omega} to $G'$. 
\end{proof}

Although we can not match the leading constant in Theorem~\ref{tri-free bound} in our ``local version,'' we can get the leading constant within a factor of $4\ln(2)$, as follows.
\begin{theorem}\label{local tri-free best constant}
  For every $\xi > 0$, if $\Delta$ is sufficiently large and $G$ is a graph of maximum degree at most $\Delta$ with correspondence assignment $(L, M)$ such that for each $v\in V(G)$,
  \begin{equation*}
    |L(v)| \geq (4 + \xi)\frac{\deg(v)}{\log_2(\deg(v))}
  \end{equation*}
  and $\deg(v) \geq \ln^2(\Delta)$, then $G$ is $(L, M)$-colorable.
\end{theorem}

\subsection{A More General Theorem}

As mentioned, Bernshteyn \cite{B17} proved that Theorems~\ref{tri-free bound} and \ref{Kr-free bound} hold for the correspondence chromatic number.  Many aspects of Bernshteyn's proofs are similar to those of Molloy's \cite{M17}; however, Bernshteyn's proof is much shorter and simpler.  Molloy used a proof technique known as ``entropy compression,'' which proves that a random algorithm terminates.  Bernshteyn cleverly realized that the use of entropy compression in Molloy's proof can be replaced with the Lopsided Lov\' asz Local Lemma, resulting in a substantial simplification of the proof.

Both proofs can be applied in the more general setting of graphs in which the average size of an independent set is somewhat large in comparison to the number of independent sets.  We make this precise by extracting a more general theorem from their proofs, and we actually prove a ``local version'' of it, as follows.

For a graph $H$, let $\avgind(H)$ and $\numind(H)$ denote the average size of an independent set and the number of independent sets in $H$ respectively.
\begin{theorem}\label{molloy blackbox}
  Let $G$ be a graph of maximum degree at most $\Delta$ with correspondence-assignment $(L, M)$, and $\varepsilon\in(0, 1/2)$.  Let $\listtarget, \listthreshold : V(G)\rightarrow\mathbb N$, and for each $v\in V(G)$, let $\indthreshold(v)$ be the minimum of $\avgind(H)$ taken over all induced subgraphs $H\subseteq G[N(v)]$ such that $\numind(H)\geq\listthreshold(v)$.
  If for each $v\in V(G)$,
  \begin{equation*}
    |L(v)| \geq \max\left\{(1 + 2\varepsilon)\frac{\deg(v)}{\indthreshold(v)}, \frac{2\listthreshold(v)\listtarget(v)}{\varepsilon(\varepsilon - 2\varepsilon^2)}\right\},
  \end{equation*}
  and
  \begin{enumerate}
  \item $\listtarget(v) \geq 18\ln(3\Delta)$,
  \item $\binom{d(v)}{\listtarget(v)}/\listtarget(v)! < \Delta^{-3}/8$,
  \end{enumerate}
  then $G$ is $(L, M)$-colorable.
\end{theorem}

We think that proving Theorem~\ref{molloy blackbox} separately makes the proof easier to understand, and we think that Theorem~\ref{molloy blackbox} may have applications not listed in this paper.

\subsection{Outline of the Paper}
In Section~\ref{mixed local section}, we prove Theorems~\ref{mixed local thm} and~\ref{local tri-free best constant} using Theorem~\ref{molloy blackbox}.  In order to apply Theorem~\ref{molloy blackbox}, one needs to find a lower bound on $\indthreshold(v)$.  We do this by proving a general bound on $\avgind(H)$ for a graph $H$ in terms of $\numind(H)$ and $\omega(H)$.  For large values of $\omega(H)$, our bound is better than the bound used by Molloy \cite{M17}, and this yields the improvement in Theorem~\ref{low omega thm}.  The condition that $|L(v)| \geq (1 + 2\varepsilon)\frac{\deg(v)}{\indthreshold(v)}$ in Theorem~\ref{mixed local thm} results in the bound $|L(v)| \geq \deg(v)\cdot f(v)$ in Theorem~\ref{mixed local thm}.  The condition that $|L(v)| \geq \frac{2\listthreshold(v)\listtarget(v)}{\varepsilon(\varepsilon - 2\varepsilon^2)}$ restricts the choice of functions $\listtarget$ and $\listthreshold$.

In Section~\ref{blackbox proof section}, we prove Theorem~\ref{molloy blackbox}.  The proof is similar to Bernshteyn's proof of Theorem~\ref{Kr-free bound} from \cite{B17}; however, we prove the more general theorem, and some changes are necessary in order to prove the ``local version'' of it.

In Section~\ref{prelim section}, we formally define correspondence coloring.  We also discuss some notation about ``partial colorings'' and some probabilistic tools that are needed in the proof of Theorem~\ref{molloy blackbox}.

%%% Local Variables:
%%% mode: latex
%%% TeX-master: "../main.tex"
%%% End:

\section{Preliminaries}
\label{prelim section}

\subsection{Correspondence Coloring}
In this subsection we define correspondence coloring.

\begin{definition}\label{correspondence coloring}
  Let $G$ be a graph with list-assignment $L$.

  \begin{itemize}
  \item If $M$ is a function defined on $E(G)$ where for each $e=uv\in E(G)$, $M_e$ is a matching of $\{u\}\times L(u)$ and $\{v\}\times L(v)$, we say $(L, M)$ is a \textit{correspondence-assignment} for $G$.
    %If for each $e=uv\in E(G)$ the matching $M_e$ saturates at least one of $\{u\}\times L(u)$ or $\{v\}\times L(v)$, then we say $(L, M)$ is \textit{total}.
	
  \item An {\em $(L, M)$-coloring} of $G$ is a function $\varphi:V(G)\rightarrow\mathbb N$ such that $\varphi(u)\in L(u)$ for every $u\in V(G)$, and for every $e=uv\in E(G)$, $(u, \varphi(u))(v, \varphi(v))\notin M_e$.  If $G$ has an $(L, M)$-coloring, then we say $G$ is {\em $(L, M)$-colorable}.
  \end{itemize}
  A correspondence-assignment $(L, M)$ is a $\textit{$k$-correspondence-assignment}$ if $L$ is a $k$-list-assignment, and the \textit{correspondence chromatic number} of $G$, denoted $\corrchi(G)$, is the minimum $k$ such that for every $k$-correspondence-assignemnt $(L, M)$, $G$ is $(L, M)$-colorable.
\end{definition}
For convenience, if $uv\in E(G)$, $c_1\in L(u)$, $c_2\in L(v)$, and $(u, c_1)(v, c_2)\in M_{uv}$, we will just say $c_1c_2\in M_{uv}$.  We will also say $c_1$ \textit{corresponds} to $c_2$.  Note that if for each $e=uv\in E(G)$ and $c\in L(u)\cap L(v)$, $cc\in M_{uv}$, then an $(L, M)$-coloring is an $L$-coloring.  Hence, for every graph $G$, $\listchi(G)\leq\corrchi(G)$.

\subsection{Partial Colorings}

The proof of Theorem~\ref{molloy blackbox} relies on analyzing a ``partial coloring'' of the graph chosen uniformly at random.  In this subsection, we define some notation that will be useful for this analysis.

\begin{definition}
  Let $G$ be a graph with correspondence-assignment $(L, M)$, and let $\blank$ be a color not in $L(v)$ for any $v\in V(G)$.  A \textit{partial $(L, M)$-coloring} of $G$ is a mapping $\partialcol$ with domain $V(G)$ such that for each $v\in V(G)$, $\partialcol(v)\in L(v)\cup\{\blank\}$.  If $\partialcol(v) = \blank$, we say $v$ is \textit{$\partialcol$-uncolored}.  For each $\partialcol$-uncolored vertex $v$, we let $L_\partialcol(v)$ denote the set of colors $c\in L(v)$ such that for every neighbor $u$ of $v$, $\partialcol(u)$ does not correspond to $c$, and we let $M_\partialcol$ denote the restriction of $M$ to edges between $\partialcol$-uncolored vertices.  If $\partialcol'$ is a partial $(L_\partialcol, M_\partialcol)$-coloring of the $\partialcol$-uncolored vertices of $G$, we let
  \begin{equation*}
    (\partialcol + \partialcol')(v) = \left\{
      \begin{array}{l l}
        \partialcol(v) & \text{if } \partialcol(v) \neq \blank,\\
        \partialcol'(v) & \text{otherwise}.
      \end{array}\right.
  \end{equation*}
\end{definition}

We will show that with nonzero probability the random partial coloring can be extended to a coloring of the whole graph.  Using the following proposition, it will suffice to show that if $\partialcol$ is a partial coloring of $G$ chosen uniformly at random, then the $\partialcol$-uncolored vertices induce a subgraph that is $(L_\partialcol, M_\partialcol)$-colorable.

\begin{proposition}\label{combining colorings prop}
  If $G$ is a graph with correspondence-assignment $(L, M)$, $\partialcol$ is a partial $(L, M)$ coloring of $G$, and $\partialcol'$ is a $(L_\partialcol, M_\partialcol)$-coloring of the graph induced by $G$ on the $\partialcol$-uncolored vertices, then $\partialcol + \partialcol'$ is an $(L, M)$-coloring of $G$.
\end{proposition}
%%% Local Variables:
%%% mode: latex
%%% TeX-master: "../main"
%%% End:

\subsection{Probabilistic Tools}

We will need the following version of the Chernoff bounds for negatively correlated random variables.  
\begin{definition}
  We say that boolean random variables $X_1, \dots, X_m$ are \textit{negatively correlated} if for every $I \subseteq \{1, \dots, m\}$,
  \begin{equation*}
    \Prob{\wedge_{i\in I}X_i} \leq \prod_{i\in I}\Prob{X_i}.
  \end{equation*}
\end{definition}
\begin{lemma}[Chernoff Bounds]\label{chernoff bounds}
  Let $X_1, \dots, X_m$ be negatively correlated boolean random variables, and let $X = \sum_{i=1}^m X_i$.  Then for any $0 < t \leq \Expect{X}$,
  \begin{equation*}
    \Prob{|X - \Expect{X}| > t} \leq 2 \exp\left(-\frac{t^2}{3\Expect{X}}\right).
  \end{equation*}
\end{lemma}

In our application of Lemma~\ref{chernoff bounds}, we will have a random partial $(L, M)$-coloring for some correspondence assignment $(L, M)$ and a boolean random variable for each color $c\in L(v)$ indicating if $v$ has a neighbor whose color corresponds to $c$.

We will also need the Lov\' asz Local Lemma.
\begin{lemma}[Lov\' asz Local Lemma]\label{local lemma}
  Let $I$ be a finite set, and for each $i\in I$, let $B_i$ be a random event.  Suppose that for every $i\in I$, there is a set $\Gamma(i)\subseteq I$ such that $|\Gamma(i)| \leq d$ and for all $Z\subseteq I\setminus \Gamma(i)$,
  \begin{equation*}
    \Prob{B_i | \bigcap_{j\in Z} \overline{B_j}} \leq p.
  \end{equation*}
  If $4pd \leq 1$, then with nonzero probability none of the events $B_i$ occur.
\end{lemma}
%%% Local Variables:
%%% mode: latex
%%% TeX-master: "../main"
%%% End:

\section{Proof of Theorem \ref{molloy blackbox}}
\label{blackbox proof section}

In this section, we prove Theorem~\ref{molloy blackbox}.  We assume $G, (L, M), \Delta, \listtarget$, and $\listthreshold$ satisfy the conditions of Theorem~\ref{molloy blackbox} throughout the section.

\subsection{Completing a partial coloring}

We prove Theorem~\ref{molloy blackbox} by finding a partial $(L, M)$-coloring of $G$ that we can greedily extend to an $(L, M$)-coloring.  The following lemma provides the existence of such a partial $(L, M)$-coloring.
\begin{lemma}\label{partial coloring lemma}
  There exists a partial $(L, M)$-coloring $\partialcol$ of $G$ such that for every $\partialcol$-uncolored vertex $v$,
  \begin{enumerate}
  \item $|L_\partialcol(v)| \geq \listtarget(v)$, and
  \item $v$ has fewer than $\listtarget(v)$ $\partialcol$-uncolored neighbors $u$ such that $\listtarget(u) \geq \listtarget(v)$.
  \end{enumerate}
\end{lemma}

Lemma~\ref{partial coloring lemma} generalizes Lemma~4.1 in the proof of Bernshteyn \cite{B17}, and the partial $(L, M)$-coloring in Lemma~\ref{partial coloring lemma} generalizes the ``flaw-free'' coloring output by the random algorithm of Molloy \cite{M17}.  When the function $\ell$ is not constant, our second condition is slightly weaker, so we are not necessarily able to complete the partial coloring greedily in any order as in their proofs.  
\begin{proof}[Proof of Theorem~\ref{molloy blackbox} assuming Lemma~\ref{partial coloring lemma}]
  Let $\partialcol$ be the partial $(L, M)$-coloring of $G$ satisfying (1) and (2) of Lemma~\ref{partial coloring lemma}.  By Proposition~\ref{combining colorings prop}, it suffices to show that the $\partialcol$-uncolored vertices induce a graph that is $(L_\partialcol, M_\partialcol)$-colorable.  This follows by ordering the $\partialcol$-uncolored vertices $v$ by $\listtarget(v)$ from greatest to least, breaking ties arbitrarily, and coloring greedily.
\end{proof}

\subsection{Analyzing a random partial coloring}

We prove Lemma~\ref{partial coloring lemma} by analyzing a partial $(L, M)$-coloring of $G$ chosen uniformly at random and using the Local Lemma to show that with nonzero probability, the random partial coloring satisfies Lemma~\ref{partial coloring lemma}.  Instead of using the Local Lemma, Molloy~\cite{M17} used the entropy compression technique.  The key insight of Bernshteyn~\cite{B17} was that a clever application of the Local Lemma is sufficient, and this greatly simplified the proof.  In order to apply the Local Lemma to prove Lemma~\ref{partial coloring lemma}, we will need the following lemma.

\begin{lemma}\label{probability lemma}
  Let $v\in V(G)$ and fix a partial $(L, M)$-coloring $\subpartialcol$ of $G - N[v]$.  Let $\nbrhoodpartialcol$ be a partial $(L_{\subpartialcol}, M_{\subpartialcol})$ coloring of $G[N(v)]$ chosen uniformly at random, and let $\partialcol = \subpartialcol + \nbrhoodpartialcol$.  Let
  \begin{enumerate}
  \item $A_{v, \subpartialcol}$ be the event that $|L_\partialcol(v)| < \listtarget(v)$ and
  \item $B_{v, \subpartialcol}$ be the event that $v$ has at least $\listtarget(v)$ $\partialcol$-uncolored neighbors $u$ such that $|L_\partialcol(u)| \geq \listtarget(u) \geq \listtarget(v)$.
  \end{enumerate}
  Then $\Prob{A_{v,\subpartialcol}}, \Prob{B_{v,\subpartialcol}} \leq \Delta^{-3}/8$.
\end{lemma}

Lemma~\ref{probability lemma} generalizes Lemma~4.2 in the proof of Bernshteyn~\cite{B17} and Lemma~12 in the proof of Molloy~\cite{M17}.

To bound $\Prob{A_{v, \subpartialcol}}$ in Lemma~\ref{probability lemma}, we show that $|L_{\partialcol}(v)|$ is large in expectation and with high probablity is concentrated around its expectation, as in the following lemma.

\begin{lemma}\label{expectation concentration lemma}
  Under the assumptions of Lemma~\ref{probability lemma},
  \begin{equation}\label{expected list size bound}
    \Expect{|L_{\partialcol}(v)|} \geq \varepsilon(\varepsilon - 2\varepsilon^2)\frac{|L(v)|}{\listthreshold(v)},
  \end{equation}
  and
  \begin{equation}\label{concentration of expected list size}
    \Prob{|L_\partialcol(v)| - \Expect{|L_{\partialcol}(v)|}| > \frac{1}{2}\Expect{|L_{\partialcol}(v)|}} \leq 2\exp\left(\frac{-\Expect{|L_\partialcol(v)|}}{12}\right).
  \end{equation}
\end{lemma}

Before proving Lemma~\ref{expectation concentration lemma}, we need some definitions.  For the remainder of this subsection, we assume $v, \subpartialcol, \nbrhoodpartialcol,$ and $\partialcol$ are as in Lemma~\ref{probability lemma}.

\begin{definition}
  For each $c\in L(v)$, let the random variable
  \begin{equation*}
    \appear_c = |\{u\in N(v) : c\nbrhoodpartialcol(u)\in M_{vu}\}|,
  \end{equation*}
  i.e.\ the number of neighbors $u$ of $v$ such that $\nbrhoodpartialcol(u)$ corresponds to $c$.
\end{definition}

Note that
\begin{equation}\label{expected list size equation}
  \Expect{|L_\partialcol(v)|} = \sum_{c\in L(v)}\Prob{\appear_c = 0}.
\end{equation}

By \eqref{expected list size equation}, in order to prove Lemma~\ref{expectation concentration lemma}, we need some bounds on $\Prob{\appear_c = 0}$ for the colors $c\in L(v)$.  These bounds will depend on the average size and number of independent sets of certain subgraphs induced by neighbors $u$ of $v$ such that $L(u)$ contains a color corresponding to $c$.  
\begin{definition}
  Fix $c\in L(v)$ and a partial $(L_{\subpartialcol}, M_{\subpartialcol})$ coloring $\nbrhoodpartialcol'$ of $G[N(v)]$ such that for no neighbor $u$ of $v$, the color $\nbrhoodpartialcol'(u)$ corresponds to $c$.  Let $\colorable(c, \nbrhoodpartialcol')$ denote the $\nbrhoodpartialcol'$-uncolored neighbors $u$ of $v$ such that $L_{\subpartialcol + \nbrhoodpartialcol'}(u)$ contains $c$, i.e.\ the $\nbrhoodpartialcol'$-uncolored neighbors of $v$ that can be colored $c$ without creating conflicts.  
\end{definition}
\begin{definition}
  For each $c\in L(v)$, let $\nbrhoodpartialcol^c$ be the partial coloring obtained from $\nbrhoodpartialcol$ by uncoloring any neighbor $u$ of $v$ such that $\nbrhoodpartialcol(u)$ corresponds to $c$.
\end{definition}
\begin{proposition}\label{ind set conditional}
  Fix $c\in L(v)$ and a partial $(L_{\subpartialcol}, M_{\subpartialcol})$-coloring $\nbrhoodpartialcol'$ such that for no neighbor $u$ of $v$, the color $\nbrhoodpartialcol'(u)$ corresponds to $c$.  Then
  \begin{enumerate}[(i)]
  \item $\Expect{\appear_c | \nbrhoodpartialcol^c = \nbrhoodpartialcol'} = \avgind(G[\colorable(c, \nbrhoodpartialcol')])$, and
  \item $\Prob{\appear_c = 0 | \nbrhoodpartialcol^c = \nbrhoodpartialcol'} = \numind(\colorable(c, \nbrhoodpartialcol'))^{-1}$.
  \end{enumerate}
\end{proposition}

\begin{definition}
  Let $\randomcolorable(c)$ denote the random set of neighbors $u$ of $v$ such that $L_{\subpartialcol}(u)$ contains $c$ and $\nbrhoodpartialcol(u)\in\{c, \blank\}$.
\end{definition}

We can now prove Lemma~\ref{expectation concentration lemma}.
\begin{proof}[Proof of Lemma~\ref{expectation concentration lemma}]
  
  First we prove that \eqref{expected list size bound} holds.  We divide $L(v)$ into two parts in the following way.  For each $c\in L(v)$, we let $c\in L_1(v)$ if $\Prob{\numind(\randomcolorable(c)) \leq \listthreshold(v)} \geq \varepsilon,$ and otherwise we let $c\in L_2(v)$.
  
  First we claim that $|L_2(v)| \leq |L(v)|/(1 + \varepsilon - 2\varepsilon^2)$.  If $c\in L_2(v)$, by Proposition~\ref{ind set conditional} and the definition of $\indthreshold(v)$,
  \begin{equation}
    \label{eq:appear-expectation-lower-bound}
    \Expect{\appear_c} \geq (1 - \varepsilon)\indthreshold(v).
  \end{equation}
  However,
  \begin{equation}
    \label{eq:appear-expectation-upper-bound}
    \sum_{c\in L_2(v)}\Expect{\appear_c} \leq \deg(v).
  \end{equation}
  By \eqref{eq:appear-expectation-lower-bound} and \eqref{eq:appear-expectation-upper-bound}, $(1 - \varepsilon)\indthreshold(v)|L_2(v)| \leq \deg(v)$.  Since $|L(v)| \geq (1 + 2\varepsilon)\frac{\deg(v)}{\indthreshold(v)}$, by rearranging terms,
  \begin{equation*}
    |L_2(v)| \leq \frac{\deg(v)}{(1 - \varepsilon)\indthreshold(v)} \leq \frac{|L(v)|}{(1 - \varepsilon)(1 + 2\varepsilon)},
  \end{equation*}
  as claimed.  Since $|L(v)| = |L_1(v)| + |L_2(v)|$, this implies $|L_1(v)| \geq (\varepsilon - 2\varepsilon^2)|L(v)|$.

  If $c\in L_1(v)$, by Proposition~\ref{ind set conditional} and the definition of $L_1(v)$,
  \begin{equation}\label{eq:L_1 appear probability}
    \Prob{\appear_c = 0} \geq \frac{\varepsilon}{\listthreshold(v)}.
  \end{equation}
  By \eqref{expected list size equation} and \eqref{eq:L_1 appear probability},
  \begin{equation*}
    \Expect{|L_\partialcol(v)|} \geq \frac{\varepsilon}{\listthreshold(v)}|L_1(v)| \geq \varepsilon(\varepsilon - 2\varepsilon^2)\frac{|L(v)|}{\listthreshold(v)},
  \end{equation*}
  as desired.

  Now we prove that \eqref{concentration of expected list size} holds.
  By \eqref{expected list size equation}, $\Expect{|L_\partialcol(v)|}$ is a sum of Boolean random variables 
  \begin{equation*}
    X_c = \left\{
      \begin{array}{l l}
        1 & \text{if } \appear_c = 0,\\
        0 & \text{if } \appear_c > 0,
      \end{array}
    \right.
  \end{equation*}
  for each $c\in L(v)$.  Since the random variables $X_c$ are negatively correlated, it follows from Lemma~\ref{chernoff bounds} with $t = \Expect{L_\partialcol(v)}/2$.
\end{proof}

We can now prove Lemma~\ref{probability lemma}.
\begin{proof}[Proof of Lemma~\ref{probability lemma}]
First we prove that $\Prob{A_{v, \subpartialcol}} \leq \Delta^{-3}/8$.  Since $|L(v)| \geq \frac{2\listtarget(v)\listthreshold(v)}{\varepsilon(\varepsilon - 2\varepsilon^2)}$, Lemma~\ref{expectation concentration lemma} implies that
  \begin{equation}\label{expected list at least twice target}
    \Expect{|L_\partialcol(v)|} \geq 2\listtarget(v).
  \end{equation}
  
  Therefore by the definition of $A_{v, \subpartialcol}$,
  \begin{equation*}\label{prob A bound}
    \Prob{A_{v,\subpartialcol}} \leq \Prob{|L_\partialcol(v)| - \Expect{|L_{\partialcol}(v)|}| > \frac{1}{2}\Expect{|L_{\partialcol}(v)|}}.
  \end{equation*}
  
  Now by Lemma~\ref{expectation concentration lemma}, \eqref{expected list at least twice target}, and the hypothesis that $\listtarget(v) \geq 18\ln(3\Delta)$,
  \begin{equation*}
    \Prob{A_{v,\subpartialcol}} \leq 2\exp\left(-\listtarget(v)/6\right) \leq \Delta^{-3}/8,
  \end{equation*}
  as desired.

  It remains to bound $\Prob{B_{v, \subpartialcol}}$.
  Let $X$ be any set of $\listtarget(v)$ neighbors $u$ of $v$ such that $|L_\partialcol(u)| \geq \listtarget(u) \geq \listtarget(v)$.  For any partial $(L_{\subpartialcol}, M_{\subpartialcol})$ coloring $\nbrhoodpartialcol'$ of $N(v)$ such that the vertices in $X$ are $\nbrhoodpartialcol'$-uncolored, there are at least $\listtarget(v)!$ partial $(L_{\subpartialcol + \nbrhoodpartialcol'}, M_{\subpartialcol + \nbrhoodpartialcol'})$ colorings of $X$, and in only one of them all of $X$ is uncolored.  Therefore the probability that every vertex of $X$ is $\partialcol$-uncolored is at most $\frac{1}{\listtarget(v)!}$.  By the Union Bound and the assumption that $\binom{\deg(v)}{\listtarget(v)}/\listtarget(v)! \leq \Delta^{-3}/8$, this implies $\Prob{B_{v, \subpartialcol}} \leq \Delta^{-3}/8$, as desired.  
\end{proof}

\subsection{Finding the partial coloring}
In this subsection, we prove Lemma~\ref{partial coloring lemma}.  Recall that Lemma~\ref{partial coloring lemma} implies Theorem~\ref{molloy blackbox}.
\begin{proof}[Proof of Lemma~\ref{partial coloring lemma}]
  Let $\partialcol$ be a partial $(L, M)$-coloring of $G$ chosen uniformly at random.  For each $v\in V(G)$, let $A_v$ be the event that $|L_\partialcol(v)| < \listtarget(v)$, let $B_v$ be the event that $v$ has at least $\listtarget(v)$ $\partialcol$-uncolored neighbors $u$ such that $|L_{\partialcol}(u)| \geq \listtarget(u) \geq \listtarget(v)$, and let $\Gamma(v)$ denote the set of vertices of distance at most three from $v$ in $G$.  Note that for all $v\in V(G)$, $|\Gamma(v)| \leq \Delta^3$.

  First we claim that with nonzero probability, none of the events $(A_v\cup B_v)$ occur in the random partial coloring $\partialcol$.  By the Local Lemma (Lemma~\ref{local lemma}), it suffices to show that for each $v\in V(G)$ and $Z\subseteq V(G)\backslash \Gamma(v)$, $\Prob{A_v | \bigcap_{u\in Z}\overline{A_u\cup B_u}} \leq \Delta^{-3}/8$ and $\Prob{B_v | \bigcap_{u\in Z}\overline{A_u\cup B_u}} \leq \Delta^{-3}/8$.  This follows from Lemma~\ref{probability lemma}.
  
  Therefore there is a partial $(L, M)$-coloring $\partialcol'$ for which none of the events $(A_v\cup B_v)$ occur.  We claim that $\partialcol'$ satisfies Lemma~\ref{partial coloring lemma}.  Suppose not.  If for some $v\in V(G)$, condition (1) is not satisfied, then $A_v$ holds, a contradiction.  Therefore we may assume for some $v\in V(G)$, condition (2) is not satisfied, that is $v$ has fewer than $\listtarget(v)$ $\partialcol'$-uncolored neighbors $u$ such that $\listtarget(u) \geq \listtarget(v)$.  Since $B_v$ holds, for some neighbor $u$, $|L_{\partialcol'}(u)| < \listtarget(u)$, contradicting that $A_u$ does not hold.  Therefore $\partialcol'$ satisfies Lemma~\ref{partial coloring lemma}, as claimed.
\end{proof}
%%% Local Variables:
%%% mode: latex
%%% TeX-master: "../main"
%%% End:

\section{Proofs of Theorems \ref{mixed local thm} and \ref{local tri-free best constant}}
\label{mixed local section}

In this section we prove Theorems~\ref{mixed local thm} and~\ref{local tri-free best constant}.  In this section, $\log$ means the base 2 logarithm.
\subsection{Bounding the Average Size of an Independent Set}

We will use Theorem~\ref{molloy blackbox}, so we will need a lower bound on $\indthreshold(v)$.  We do this by bounding the average size of an independent set in terms of the total number.  In the proof of Theorem~\ref{Kr-free bound}, Molloy~\cite{M17} and Bernshteyn~\cite{B17} use the following result of Shearer~\cite{Sh95}, which we will also need.

\begin{lemma}[\cite{Sh95}]\label{shearers bound}
  If $H$ is a graph with no clique of size greater than $\omega$, then
  \begin{equation*}
    \avgind(H) \geq \frac{\log(\numind(H))}{2\omega\log(\log(\numind(H)))}.
  \end{equation*}
\end{lemma}

We will also need the following result of Alon \cite{A96}.
\begin{lemma}[\cite{A96}]\label{alons bound} %lemma 2.2
  If $H$ is a graph on $n$ vertices, then
  \begin{equation*}    
  \avgind(H) \geq \frac{\log(\numind(H))}{10\log\left(n/\log(\numind(H)) + 1\right)}.
\end{equation*}
\end{lemma}

Since $\log(\numind(H)) \geq \alpha(H)$, we can replace the $\log(\numind(H))$ in the denominator of the bound in Lemma~\ref{alons bound} with $\alpha(H)$ to get a suitable bound if $H$ contains a large independent set.

The following lemma provides an improvement over Lemma \ref{shearers bound} for larger values of $\omega$.

\begin{lemma}\label{average independent set size bound}
  If $H$ is a graph on $n$ vertices with no clique of size greater than $\omega$ and $n$ is sufficiently large, then
  \begin{equation*}
    \avgind(H) \geq \frac{1}{24}\sqrt{\frac{\log(\numind(H))}{\log(\omega)}}.
  \end{equation*}
\end{lemma}

We will actually use Lemma~\ref{alons bound} to prove Lemma~\ref{average independent set size bound}.  To apply Lemma~\ref{alons bound}, we need to upper bound $\log(n)$ in terms of $\log(\numind(H))$ and $\omega(H)$, as in the following lemma.

\begin{lemma}\label{number of independent set size bound}
  If $H$ is a graph on $n$ vertices with no clique of size greater than $\omega$ and $n$ is sufficiently large, then
  \begin{equation*}
    \log(\numind(H)) \geq \frac{\log^2(n)}{2e\log(\omega)}.
  \end{equation*}
\end{lemma}
\begin{proof}
  We may assume $\omega\geq 3$, or else $H$ has an independent set of size at least $\sqrt{n}$, and the result follows.
  
  Let $\alpha$ be some positive integer to be determined later, and let $s = R(\alpha, \omega + 1)$, the Ramsey number.  We will actually prove there are at least $2^\frac{\log^2(n)}{2e\log(\omega)}$ independent sets of size $\alpha$.
  
  By the definition of $s$, every subset of $V(H)$ of size $s$ has an independent set of size $\alpha$.  Since every independent set of size $\alpha$ is contained in at most $\binom{n - \alpha}{s - \alpha}$ subsets of $V(H)$ of size $s$, there are at least
  \begin{equation}\label{alpha ind sets}
    \binom{n}{s}\Big/\binom{n - \alpha}{s - \alpha} \geq \left(\frac{n - \alpha}{s}\right)^\alpha = 2^{\alpha(\log(n - \alpha) - \log(s))}
  \end{equation}
  independent sets of size $\alpha$.  

  We let $\alpha = \frac{\log(n)}{e\log(\omega)} + 1$.  By \eqref{alpha ind sets}, it suffices to show that $\log(n - \alpha) - \log(s) \geq \log(n)/2$.  It is well-known that $R(\alpha, \omega + 1) \leq \binom{\alpha + \omega - 1}{\alpha - 1} \leq \left(\frac{\alpha + \omega - 1}{\alpha - 1}\cdot e\right)^{\alpha - 1}$.  Therefore
  \begin{equation*}
    \log(s) \leq (\alpha - 1)\log\left(\frac{\alpha + \omega - 1}{\alpha - 1}\cdot e\right) \leq \frac{\log(n)}{e\log(\omega)}\log\left(e + \omega\frac{\log(\omega)}{\log(n)}\right).
  \end{equation*}
  Since $\alpha = o(n)$ and $\omega\geq 3$, for $n$ sufficiently large, $\log(n - \alpha) - \log(s) \geq \log(n) / 2$, as desired.
\end{proof}

Now we can prove Lemma~\ref{average independent set size bound}.
\begin{proof}[Proof of Lemma~\ref{average independent set size bound}]
  By Lemma~\ref{alons bound},
  \begin{equation*}
    \avgind(H) \geq \frac{\log(\numind(H))}{10\log(n)}.
  \end{equation*}
  By Lemma~\ref{number of independent set size bound},
  \begin{equation*}
    \log(n) \leq \sqrt{2e\log(\numind(H))\log(\omega)},
  \end{equation*}
  and the result follows.
\end{proof}

\subsection{The Proofs}
Lemmas~\ref{shearers bound}, \ref{alons bound}, and \ref{average independent set size bound} provide good enough bounds for $\indthreshold(v)$.  Now we are able to prove Theorems~\ref{mixed local thm} and~\ref{local tri-free best constant}.
\begin{proof}[Proof of Theorem~\ref{mixed local thm}]
  Let $\varepsilon = 1/4$, and let $v\in V(G)$.  We show that $v$ satisfies the conditions of Theorem~\ref{molloy blackbox}.  Let $\listtarget(v) = \deg(v)^{5/8}$, and let $\listthreshold(v) = \deg(v)^{1/4}$.  By Lemma~\ref{average independent set size bound},
  \begin{equation*}
    \indthreshold(v) \geq \frac{1}{48}\sqrt{\frac{\log(\deg(v))}{\log(\omega(v))}}.
  \end{equation*}
  By Lemma~\ref{shearers bound},
  \begin{equation*}
    \indthreshold(v) \geq \frac{\log(\deg(v))}{8\omega(v)\log(\log(\deg(v)))}.
  \end{equation*}
  Note that $\log(\numind(H)) \geq \alpha(H)$ for any graph $H$.  Hence if $H\subseteq G[N(v)]$, then $|V(H)|/\log(\numind(H)) \leq \chi(H) \leq \chi(v)$.  Therefore by Lemma~\ref{alons bound},
  \begin{equation*}
    \indthreshold(v) \geq \frac{\log(\deg(v))}{40\log(\chi(v) + 1)}.
  \end{equation*}
  Since $|L(v)| \geq 72\deg(v)f(v)$, it follows that $|L(v)| \geq (1 + 2\varepsilon)\frac{\deg(v)}{\indthreshold(v)}$, as desired.
  Note that $f(v)^8 = o(\deg(v))$.
  Since $\Delta$ is sufficiently large and $\deg(v) \geq \log^2(\Delta)$, we may assume $f(v) \geq 168\deg(v)^{-1/8}$.  Since $\listthreshold(v)\listtarget(v) = \deg(v)^{7/8}$, $|L(v)| \geq \frac{2\listthreshold(v)\listtarget(v)}{\varepsilon(\varepsilon - 2\varepsilon^2)}$, as desired.

  Since $\listtarget(v) \geq \ln^{5/4}(\Delta)$ and $\Delta$ is sufficiently large, $\listtarget(v) \geq 18\ln(3\Delta)$, as desired.  It remains to show that $\binom{\deg(v)}{\listtarget(v)}/\listtarget(v)! < \Delta^{-3}/8$.  We will use the following form of Stirling's approximation:
  \begin{equation*}
    n! \geq \sqrt{2\pi}n^{n + 1/2}e^{-n}.
  \end{equation*}
  Therefore
  \begin{equation}\label{binomial list threshold bound}
    \binom{\deg(v)}{\listtarget(v)}/ \listtarget(v)! < \frac{\deg(v)^{\listtarget(v)}}{(\listtarget(v)!)^2}  \leq \left(\frac{e^2\deg(v)}{2\pi\listtarget(v)^{2 + 1/\listtarget(v)}}\right)^{\listtarget(v)}.
  \end{equation}
  Since $\deg(v)/\listtarget(v)^{2 + 1/\listtarget(v)} \leq \listtarget(v)^{-2/5}$, by taking the logarithm of the bound in \eqref{binomial list threshold bound}, it suffices to show that
  \begin{equation*}
    \listtarget(v)\ln\left(\frac{2\pi\listtarget(v)^{2/5}}{e^2}\right) \geq 3\ln(8\Delta).
  \end{equation*}
  Since $\listtarget(v) \geq \ln^{5/4}(\Delta)$ and $\Delta$ is sufficiently large, this follows.
\end{proof}

\begin{proof}[Proof of Theorem~\ref{local tri-free best constant}]
  Let $\varepsilon = \xi/10$ and $\varepsilon' > 0$ be some constant to be chosen later, and let $v\in V(G)$.
  Let $\listtarget(v) = \deg(v)^{(1 + \varepsilon')/2}$, and let $\listthreshold(v) = \deg(v)^{(1 - 2\varepsilon')/2}$.  Since $G$ is triangle-free,
  \begin{equation*}
    \indthreshold(v) \geq \frac{\log(\listthreshold(v))}{2} = (1 - 2\varepsilon')\log(\deg(v))/4.
  \end{equation*}
  Since $|L(v)| \geq (4 + \xi)\deg(v)/\log(\deg(v))$, it follows that $|L(v)| \geq (1 + \xi/4)(1 - 2\varepsilon')\frac{\deg(v)}{\indthreshold(v)}$.  We choose $\varepsilon'$ sufficiently small so that $(1 + \xi/4)(1 - 2\varepsilon') \geq 1 + 2\varepsilon$.

  Note that $\listthreshold(v)\listtarget(v) = \deg(v)^{1-\varepsilon'/2}$.  Hence we may assume $\deg(v)$ is sufficiently large so that $|L(v)| \geq \frac{2\listthreshold(v)\listtarget(v)}{\varepsilon(2 - 2\varepsilon^2)}$, as desired.

  Note also that $\listtarget(v) \geq \ln^{1 + \varepsilon'}(\Delta)$.  The rest of the proof is similar to the proof of Theorem~\ref{mixed local thm}, so we omit it.
\end{proof}
%%% Local Variables:
%%% mode: latex
%%% TeX-master: "../main.tex"
%%% End:

%%%%%%%%%%% bibliography %%%%%%%%%%%%
\bibliographystyle{plain}
\bibliography{low-omega}

\begin{thebibliography}{10}

\bibitem{AKS80}
M.~Ajtai, J.~Koml\'os, and E.~Szemer\'edi.
\newblock A note on {R}amsey numbers.
\newblock {\em J. Combin. Theory Ser. A}, 29(3):354--360, 1980.

\bibitem{A96}
N.~Alon.
\newblock Independence numbers of locally sparse graphs and a {R}amsey type
  problem.
\newblock {\em Random Structures Algorithms}, 9(3):271--278, 1996.

\bibitem{B17}
A.~{Bernshteyn}.
\newblock {The Johansson--Molloy Theorem for DP-Coloring}.
\newblock {\em ArXiv e-prints}, August 2017.

\bibitem{B09}
T.~Bohman.
\newblock The triangle-free process.
\newblock {\em Adv. Math.}, 221(5):1653--1677, 2009.

\bibitem{BK10}
T.~Bohman and P.~Keevash.
\newblock The early evolution of the {$H$}-free process.
\newblock {\em Invent. Math.}, 181(2):291--336, 2010.

\bibitem{BPP17}
M.~Bonamy, T.~Perrett, and L.~Postle.
\newblock Colouring graphs with sparse neighbourhoods: Bounds and applications.
\newblock submitted.

\bibitem{DP17}
M.~Delcourt and L.~Postle.
\newblock On the list coloring version of reed's conjecture.
\newblock manuscript.

\bibitem{DP15}
Z.~Dvo\v{r}\'a{k} and L.~Postle.
\newblock Correspondence coloring and its application to list-coloring planar
  graphs without cycles of lengths 4 to 8.
\newblock {\em J. Combin. Theory Ser. B}, 129:38--54, 2018.

\bibitem{ERT79}
P.~Erd\H{o}s, A.~L. Rubin, and H.~Taylor.
\newblock Choosability in graphs.
\newblock In {\em Proceedings of the {W}est {C}oast {C}onference on
  {C}ombinatorics, {G}raph {T}heory and {C}omputing ({H}umboldt {S}tate
  {U}niv., {A}rcata, {C}alif., 1979)}, Congress. Numer., XXVI, pages 125--157.
  Utilitas Math., Winnipeg, Man., 1980.

\bibitem{J96-tri}
A.~Johansson.
\newblock Asymptotic choice number for triangle free graphs.
\newblock Unpublished Manuscript, 1996.

\bibitem{J96-Kr}
A.~Johansson.
\newblock The choice number of sparse graphs.
\newblock Unpublished Manuscript, 1996.

\bibitem{K95}
J.~H. Kim.
\newblock On {B}rooks' theorem for sparse graphs.
\newblock {\em Combin. Probab. Comput.}, 4(2):97--132, 1995.

\bibitem{K95ramsey}
J.~H. Kim.
\newblock The {R}amsey number {$R(3,t)$} has order of magnitude {$t^2/\log t$}.
\newblock {\em Random Structures Algorithms}, 7(3):173--207, 1995.

\bibitem{M17}
M.~{Molloy}.
\newblock {The list chromatic number of graphs with small clique number}.
\newblock {\em ArXiv e-prints}, January 2017.

\bibitem{PS15}
S.~Pettie and H-H. Su.
\newblock Distributed coloring algorithms for triangle-free graphs.
\newblock {\em Inform. and Comput.}, 243:263--280, 2015.

\bibitem{R98}
B.~Reed.
\newblock {$\omega,\ \Delta$}, and {$\chi$}.
\newblock {\em J. Graph Theory}, 27(4):177--212, 1998.

\bibitem{Sh95}
J.~B. Shearer.
\newblock On the independence number of sparse graphs.
\newblock {\em Random Structures Algorithms}, 7(3):269--271, 1995.

\bibitem{S77}
J.~Spencer.
\newblock Asymptotic lower bounds for {R}amsey functions.
\newblock {\em Discrete Math.}, 20(1):69--76, 1977/78.

\end{thebibliography}

\end{document}